\documentclass[10pt]{amsart}

\usepackage{amsmath}
\usepackage{amssymb}
\usepackage{graphicx}
\usepackage{amsfonts}
\usepackage{mathrsfs}
\usepackage[english]{babel}

\newtheorem{theorem}{Theorem}[section]
\newtheorem{corollary}[theorem]{Corollary}
\newtheorem{lemma}[theorem]{Lemma}
\newtheorem{proposition}[theorem]{Proposition}
\theoremstyle{definition}

\newtheorem{example}[theorem]{Example}
\newtheorem{remark}[theorem]{Remark}

\numberwithin{equation}{section}

\title[First-order approximations of strong vector equilibria]
{First-order approximation of strong vector equilibria with application
to nondifferentiable constrained optimization}

\author[A. Uderzo]{Amos Uderzo}

\address[A. Uderzo]{Dept. of Mathematics and Applications, University
of Milano - Bicocca, Milano, Italy}
\email{{\tt amos.uderzo@unimib.it}}

\keywords{Strong vector equilibrium, contingent cone, nondifferentiable optimization,
generalized differentiation, subdifferential, mathematical programming with
equilibrium constraint}

\subjclass[2010]{49J53, 49J52, 90C33}

\date{\today}

\newcommand{\R}{\mathbb R}
\newcommand{\N}{\mathbb N}

\newcommand{\X}{\mathbb X}

\newcommand{\Uball}{{\mathbb B}}
\newcommand{\Usfer}{{\mathbb S}}

\newcommand{\nullv}{\mathbf{0}}

\newcommand{\cl}{{\rm cl}\, }
\newcommand{\cone}{{\rm cone}\, }
\newcommand{\bd}{{\rm bd}\, }
\newcommand{\inte}{{\rm int}\, }

\newcommand{\Lin}{\mathscr{L}}
\newcommand{\PH}{\mathscr{P}\hskip-0.1cm\mathscr{H}}
\newcommand{\Equi}{{\mathcal S}{\mathcal E}}

\newcommand{\MPVEC}{{\rm MPVEC}}
\newcommand{\VEP}{{\rm VEP}}

\newcommand{\parord}{\le_{{}_C}}

\newcommand{\Upsubd}{\widehat{\partial}^+}
\newcommand{\Fsubd}{\widehat{\partial}}
\newcommand{\Csubd}{{\partial}_C}

\newcommand{\dcone}[1]{{#1}^{{}^\ominus}}   

\newcommand{\ball}[2]{{\rm B}\left(#1;#2\right)}

\newcommand{\dist}[2]{{\rm dist}\left(#1;#2\right)}
\newcommand{\exc}[2]{{\rm exc}(#1;#2)}

\newcommand{\Tang}[2]{{\rm T}(#1;#2)}    
\newcommand{\Wang}[2]{{\rm T}_{\rm r}(#1;#2)}  
\newcommand{\Fang}[2]{{\rm T}_{\rm f}(#1;#2)}  
\newcommand{\Iang}[2]{{\rm I}(#1;#2)}  
\newcommand{\Ncone}[2]{{\rm N}(#1;#2)}

\newcommand{\Bder}[2]{{\rm D}_B#1(#2)}

\newcommand{\DHuder}[3]{{\rm D}^+_H#1(#2;#3)}
\newcommand{\Duder}[3]{{\rm D}^+_D#1(#2;#3)}
\newcommand{\DHlder}[3]{{\rm D}^-_H#1(#2;#3)}
\newcommand{\Dlder}[3]{{\rm D}^-_D#1(#2;#3)}

\begin{document}

\begin{abstract}
Vector equilibrium problems are a natural generalization to the
context of partially ordered spaces of the Ky Fan inequality, where
scalar bifunctions are replaced with vector bifunctions.
In the present paper, the local geometry of the strong solution
set to these problems is investigated through its inner/outer
conical approximations.
Formulae for approximating the contingent cone
to the set of strong vector equilibria are established, which are
expressed via Bouligand derivatives of the bifunctions.
These results are subsequently employed for deriving both necessary
and sufficient optimality conditions for problems, whose feasible region
is the strong solution set to a vector equilibrium problem, so they
can be cast in mathematical programming with equilibrium constraints.
\end{abstract}

\maketitle




\section{Introduction}

Given a mapping (vector-valued bifunction)
$f:\R^n\times\R^n\longrightarrow\R^m$, with $\R^m$
being partially ordered by a (nontrivial) closed, convex
and pointed cone $C\subset\R^m$, and a nonempty, closed set
$K\subseteq\R^n$, by strong vector equilibrium problem
the problem is meant
$$
  \hbox{ find $x\in K$ such that } f(x,z)\in C,
  \quad\forall z\in K.  \leqno  (\VEP)
$$
The set of all solutions (if any) to problem $(\VEP)$ will be
denoted throughout the paper by $\Equi$, namely
\begin{equation}     \label{eq:interEquiref}
   \Equi=\bigcap_{z\in K}f^{-1}(\cdot,z)(C)\cap K,
\end{equation}
and referred to as the set of strong vector equilibria.
Clearly, strong vector equilibrium problems are a natural
generalization of the well-known Ky Fan inequality to
the more general context of partially ordered vector spaces.
Similarly as their scalar counterpart, they provide a convenient
format to treat in an unifying framework several different classes
of problems, ranging from multicriteria optimization problems,
vector Nash equilibrium problems,
to vector variational inequalities and complementarity problems
(see, for instance, \cite{AnKoYa02,AnKoYa18,AnOeSc97,BiHaSc97,Gong06,
GoKiYa08,Oett97}).

As for many problems formalized by traditional or generalized equations,
for several purposes the mere knowledge of a single solution to
$(\VEP)$ is not enough. Very often, once a strong vector equilibrium $\bar x\in\Equi$
has been found (or shown to exist), one would need/aspire to glean insights
into the behaviour of the set $\Equi$ around $\bar x$.
The fact that $\bar x$ may be an isolated element of $\Equi$ or lie in
the boundary or, instead, be an interior element of this set, might
change dramatically the outcome of a further analysis, where the
local geometry of $\Equi$ around $\bar x$ does matter.
On the other hand, finding all the solutions of $(\VEP)$ around $\bar x$
could be a task that one can hardly accomplish in many concrete cases.
What is reasonably achievable sometimes is only a local approximation
of $\Equi$ near $\bar x$, yet suitable in specific circumstances.
To mention one of them, with connection with the subject of the
present paper, consider the successful approach to optimality conditions
for constrained problems, where at a certain step
an approximated representation of the feasible region already does the trick.

It is well known that in nonsmooth analysis tangent cones, working
as a surrogate of derivative for sets, are the main tools for
formalizing first-order (and beyond, if needed) approximations
of sets. So the main aim of the present paper is to provide
elements for a conical approximation of strong vector equilibria.
It should be remarked that a difficulty in undertaking such a task
comes from the fact that the set $\Equi$ is not explicitly defined.
Besides, if addressing this question through the reformulation of
$\Equi$ as in $(\ref{eq:interEquiref})$, classical results on the
tangent cone representation of such sets as $f^{-1}(\cdot,z)(C)\cap K$,
now at disposal in nonsmooth analysis as a modern development of
the Lyusternik theorem (see \cite{Mord06a,Mord18,Schi07}), seem not
be readily exploitable because of the intersection over $K$ appearing
in $(\ref{eq:interEquiref})$.

In this context, the findings exposed in what follows are focussed
on representing the contingent cone to $\Equi$ at a given strong vector
equilibrium $\bar x$, which is one of the most employed conical approximations
in the literature devoted to variational analysis and optimization.
The representation of such a cone will be performed by means of
first-order approximations of the problem data, namely generalized
derivatives of the bifunction $f$ and tangent cones of the set
$K$ defining $(\VEP)$. In other words, following a principle deep-rooted
in many contexts of nonlinear analysis, approximations of the solution
set to a given problem are obtained by means of exact solutions to approximated
problems.

The paper is structured as follows. Section \ref{Sect:2} aims at recalling preliminary
notions of nonsmooth analysis, which play a role in formulating and
establishing the achievements of the paper.
Section \ref{Sect:3} contains the main results concerning the first-order
approximation of the contingent cone to $\Equi$.
In Section \ref{Sect:4}, these results are applied to derive both necessary
and sufficient optimality conditions for nondifferentiable optimization problems,
whose constraint systems are formalized as a strong vector equilibrium problem.

Below, the basic notations employed in the paper are listed.
The acronyms l.s.c., u.s.c and p.h. stand for lower semicontinuous, upper
semicontinuous and positively homogeneous, respectively.
$\R^d$ denotes the finite-dimensional Euclidean space, with dimension $d\in\N$.
The closed  ball centered at an element $x\in\R^d$, with radius $r\ge 0$, is
denoted by $\ball{x}{r}$. In particular, $\Uball=\ball{\nullv}{1}$ stands for
the unit ball, whereas $\Usfer$ stands for the unit sphere, $\nullv$ denoting
the null vector of an Euclidean space.
Given a subset $S\subseteq\R^d$, the distance of a point $x$ from a set $S$ is denoted
by $\dist{x}{S}$, with the convention that $\dist{x}{\varnothing}=+\infty$.
The prefix $\inte S$ denotes the interior of $S$, $\cl S$ denotes its closure, whereas
$\cone S$ its conical hull, respectively. Given two subsets $A$ and $B$ of the same
space, the excess of $A$ over $B$ is indicated by $\exc{A}{B}=
\sup_{a\in A}\dist{a}{B}$.
By $\PH(\R^n,\R^m)$ the space of all continuous p.h. mappings
acting between $\R^n$ and $\R^m$ is denoted, equipped with the norm
$\|h\|_{\PH}=\sup_{u\in\Usfer}\|h(u)\|$, $h\in\PH(\R^n,\R^m)$, while
$\Lin(\R^n,\R^m)$ denotes its subspace of all linear
operators.
The inner product of an Euclidean space will be denoted
by $\langle\cdot,\cdot\rangle$.
Whenever $C$ is a cone in $\R^n$, by $\dcone{C}=\{v\in\R^n:\ \langle v,
c\rangle\le 0,\quad\forall c\in C\}$ the negative dual (a.k.a. polar)
cone to $C$ is denoted.
Given a function $\varphi:\X\longrightarrow\R\cup\{\pm\infty\}$,
the symbol $\partial\varphi(x)$ denotes the subdifferential
of $\varphi$ at $x$ in the sense of convex analysis (a.k.a. Fenchel
subdifferential). The normal cone to a set $S\subseteq\R^q$ at $x\in S$
in the sense of convex analysis is denoted by $\Ncone{x}{S}=\{v\in\R^n:\
\langle v,s-x\rangle,\ \forall s\in S\}$.

\vskip1cm


\section{Preliminaries}    \label{Sect:2}

\subsection{Approximation of sets}

Given a nonempty set $K\subseteq\R^n$ and $\bar x\in K$, in the sequel
the following different notions of tangent cone will be mainly employed:
\begin{itemize}

\item[(i)] the contingent (a.k.a. Bouligand tangent) cone to $K$ at $\bar x$,
which is defined by
$$
  \Tang{\bar x}{K}=\{v\in\R^n:\ \exists (v_n)_n,\ v_n\to v,\
  \exists (t_n)_n,\ t_n\downarrow 0:\ \bar x+t_nv_n\in K,\
  \forall n\in \N\};
$$

\item[(ii)] the cone of radial (a.k.a. weak feasible) directions to $K$ at $\bar x$,
which is defined by
$$
  \Wang{\bar x}{K}=\{v\in\R^n:\ \forall\epsilon>0\
  \exists t_\epsilon\in (0,\epsilon):\ \bar x+t_\epsilon v
  \in K\}.
$$
\end{itemize}
Clearly, for every $K\subseteq\R^n$ and $\bar x\in K$,
it is $\Wang{\bar x}{K}\subseteq\Tang{\bar x}{K}$. Moreover $\Tang{\bar x}{K}$
is always closed.
If, in particular, $K$ is convex, then the following representations hold
\begin{equation}     \label{eq:convexWTangcone}
  \Wang{\bar x}{K}=\cone(K-\bar x) \quad\hbox{ and }\quad
  \Tang{\bar x}{K}=\cl(\cone(K-\bar x))=\cl\Wang{\bar x}{K}
\end{equation}
(see \cite[Proposition 11.1.2(d)]{Schi07}). Thus, in such an event, both
$\Wang{\bar x}{K}$ and $\Tang{\bar x}{K}$ are convex.
It is well known that an equivalent (variational) reformulation of the notion of
contingent cone is provided by the equality
\begin{equation}     \label{eq:charTangcone}
  \Tang{\bar x}{K}=\left\{v\in\R^n:\ \liminf_{t\downarrow 0}
  {\dist{\bar x+tv}{K}\over t}=0\right\}.
\end{equation}

\begin{remark}    \label{rem:polyWang}
Whenever a convex set $K\subseteq\R^n$ is, in particular, polyhedral,
one has $\Wang{\bar x}{K}=\Tang{\bar x}{K}$. To see this, it
suffices to exploit the formulae in $(\ref{eq:convexWTangcone})$
and to observe that, in the present circumstance, $\Wang{\bar x}{K}$
happens to be closed. The latter follows from the fact that, if $S$
is a closed affine half-space in $\R^n$, then $\Wang{\bar x}{S}=
\cone(S-\bar x)=S-\bar x$ is a closed set and from the fact that,
if $K_1$ and $K_2$ are convex sets with $\bar x\in\ K_1\cap K_2$,
then it holds $\Wang{\bar x}{K_1\cap K_2}=\Wang{\bar x}{K_1}\cap
\Wang{\bar x}{K_2}$.
\end{remark}

Along with the above cones, in the context of optimization problems
some further notions of first-order conical approximation will be needed:
\begin{itemize}

\item[(iii)] the cone of radial inner (a.k.a. feasible) directions
to $K$ at $\bar x$, which is defined by
$$
  \Fang{\bar x}{K}=\{v\in\R^n:\ \exists\epsilon>0:\ \forall
  t\in (0,\epsilon),\ \bar x+tv\in K\};
$$

\item[(vi)] the cone of inner directions (a.k.a. interior displacements)
to $K$ at $\bar x$, which is defined by
$$
  \Iang{\bar x}{K}=\{v\in\R^n:\ \exists\epsilon>0:\ \forall
  u\in\ball{v}{\epsilon},\ \forall t\in (0,\epsilon),\ \bar x+tu\in K\}.
$$
\end{itemize}

For a systematic discussion about properties of the above tangent
cones and their relationships, the reader is referred for instance to
\cite[Chapter 4]{AubFra90}, \cite[Chapter I.1]{DemRub95}, \cite{ElsThi88},
\cite[Chapter 2]{Peno13}, and \cite[Chapter 11]{Schi07}.

\subsection{Approximation of scalar functions}

Given a function $\varphi:\R^n\longrightarrow\R\cup\{\pm\infty\}$,
let $\bar x\in\varphi^{-1}(\R)$. The set
$$
   \Upsubd\varphi(\bar x)=\left\{v\in\R^n:\ \limsup_{x\to\bar x}
   {\varphi(x)-\varphi(\bar x)-\langle v,x-\bar x\rangle\over\|x-\bar x\|}
   \le 0\right\}
$$
is called (Fr\'echet) upper subdifferential of $\varphi$ at $\bar x$.
Any element $v\in\Upsubd\varphi(\bar x)$ can be characterized by the
existence of a function $\psi:\R^n\longrightarrow\R$ such that
$\varphi(\bar x)=\psi(\bar x)$, $\varphi(x)\le\psi(x)$, for
every $x\in\R^n$, $\psi$ is (Fr\'echet) differentiable at $\bar x$
and $v=\nabla\psi(\bar x)$. If $\varphi:\R^n\longrightarrow\R$
is concave, then $\Upsubd\varphi(\bar x)$ coincides with the
superdifferential (a.k.a. upper subdifferential) in the sense of
convex analysis, i.e. $-\partial(-\varphi)(\bar x)$.

Whenever $\varphi$ is an u.s.c. function, the upper subdifferential
admits another characterization in terms of Dini-Hadamard directional
derivative, in fact being equivalent to the Dini-Hadamard upper
subdifferential (in finite-dimensional spaces, the Fr\'echet bornology
is equivalent to the Hadamard bornology). More precisely, it holds
\begin{equation}   \label{eq:UpsubdDHder}
  \Upsubd\varphi(\bar x)=\{v\in\R^n:\ \langle v,w\rangle\ge
  \DHuder{\varphi}{\bar x}{w},\quad\forall w\in\R^n\},
\end{equation}
where
$$
  \DHuder{\varphi}{\bar x}{w}=\limsup_{u\to w\atop t\downarrow 0}
  {\varphi(\bar x+tu)-\varphi(\bar x)\over t}
$$
denotes the Dini-Hadamard upper directional derivative of $\varphi$
at $\bar x$, in the direction $w\in\R^n$ (see \cite[Chapter 1.3]{Mord18},
\cite[Chapter 8.B]{RocWet98}).
Let us recall that, whenever $\varphi$ is locally Lipschitz around $\bar x$,
its Dini-Hadamard directional derivative at $\bar x$ takes the following simpler
form
$$
  \Duder{\varphi}{\bar x}{w}=\limsup_{t\downarrow 0}
  {\varphi(\bar x+tw)-\varphi(\bar x)\over t},
$$
which is known as Dini upper directional derivative. The lower versions
of these generalized derivatives are
$$
  \DHlder{\varphi}{\bar x}{w}=\liminf_{u\to w\atop t\downarrow 0}
  {\varphi(\bar x+tu)-\varphi(\bar x)\over t},
$$
called the Dini-Hadamard lower directional (a.k.a. contingent) derivative
of $\varphi$ at $\bar x$, in the direction $w$, and
$$
  \Dlder{\varphi}{\bar x}{w}=\liminf_{t\downarrow 0}
  {\varphi(\bar x+tw)-\varphi(\bar x)\over t},
$$
called the Dini lower directional derivative of $\varphi$
at $\bar x$, in the direction $w$.

The set
$$
   \Fsubd\varphi(\bar x)=\left\{v\in\R^n:\ \liminf_{x\to\bar x}
   {\varphi(x)-\varphi(\bar x)-\langle v,x-\bar x\rangle\over\|x-\bar x\|}
   \ge 0\right\}
$$
is called (Fr\'echet) regular subdifferential of $\varphi$ at $\bar x$.
Whenever $\varphi$ is l.s.c. around $\bar x$, it admits the following representation
in terms of Dini-Hadamard lower directional generalized derivative
\begin{equation}   \label{eq:FsubdDHder}
  \Fsubd\varphi(\bar x)=\{v\in\R^n:\ \langle v,w\rangle\le
  \DHlder{\varphi}{\bar x}{w},\quad\forall w\in\R^n\}.
\end{equation}
Whenever $\varphi$ is Fr\'echet differentiable at $\bar x$, one has
$\Upsubd\varphi(\bar x)=\Fsubd\varphi(\bar x)=\{\nabla\varphi(\bar x)\}$,
where $\nabla\varphi(\bar x)$ denotes the gradient of $\varphi$ at
$\bar x$.

Comprehensive discussions from various viewpoints as well
as detailed material about these generalized derivatives can
be found in many textbooks devoted to nonsmooth analysis,
among which \cite[Chapter I.1]{DemRub95}, \cite[Chapter 1]{Mord18},
\cite[Chapter 2]{Peno13}, \cite[Chapter 8]{RocWet98}, \cite{Schi07}.

\subsection{Approximation of mappings and bifunctions}

A mapping $g:\R^n\longrightarrow\R^m$ is said to be $B$-differentiable
at $\bar x\in\R^n$ if there exists a mapping $\Bder{g}{\bar x}\in\PH(\R^n,\R^m)$
such that
$$
  \lim_{x\to\bar x}{\|g(x)-g(\bar x)-\Bder{g}{\bar x}(x-\bar x)\|
  \over\|x-\bar x\|}=0.
$$
As a consequence of the continuity of $\Bder{g}{\bar x}$, it is readily
seen that if $g$ is $B$-differentiable at $\bar x$, it is also continuous
at the same point.
Notice that, when, in particular, $\Bder{g}{\bar x}\in\Lin(\R^n,\R^m)$, $g$ turns out to be (Fr\'echet)
differentiable at $\bar x$. In such an event, its derivative, represented by its
Jacobian matrix, will be indicated by $\nabla g(\bar x)$.
Given a nonempty set $K\subseteq\R^n$, a bifunction $f:\R^n\times\R^n
\longrightarrow\R^m$ is said to be $B$-differentiable at $\bar x\in K$,
uniformly on $K$, if there exists a family $\{\Bder{f}{\bar x,z}\in
\PH(\R^n,\R^m):\ z\in K\}$ such that for every $\epsilon>0$ $\exists\delta_\epsilon>0$
such that
$$
  \sup_{z\in K}{\|f(x,z)-f(\bar x,z)-\Bder{f}{\bar x,z}(x-\bar x)\|
  \over\|x-\bar x\|}<\epsilon,\quad\forall x\in\ball{\bar x}{\delta_\epsilon}.
$$
It should be clear that the above notion of generalized differentiation for bifunctions
is a kind of partial differentiation, in considering variations of a mapping with respect
to changes of one variable only.

\begin{example}    \label{ex:unifBdiff3}
(i) Separable mappings: let us consider mappings $f:\R^n\times\R^n\longrightarrow\R^m$,
which can be expressed in the form
$$
  f(x,z)=f_1(x)+f_2(z),
$$
for proper $f_1,\, f_2:\R^n\longrightarrow\R^m$. Whenever $f_1$ is $B$-differentiable at $\bar x$,
with $B$-derivative $\Bder{f_1}{\bar x}$, the bifunction $f$ is $B$-differentiable at $\bar x$
uniformly on $K$, with $\{\Bder{f}{\bar x,z}:\ z\in K\}=\{\Bder{f_1}{\bar x}\}$.

(ii) Factorable mappings: whenever a mapping $f:\R^n\times\R^n\longrightarrow\R^m$
can be factorized as
$$
  f(x,z)=\alpha(z)g(x),
$$
where $g:\R^n\longrightarrow\R^m$ is $B$-differentiable at $\bar x$, with $B$-derivative
$\Bder{g}{\bar x}$, and $\alpha:\R^n\longrightarrow\R$ is bounded on $K$, the bifunction $f$
is $B$-differentiable at $\bar x$ uniformly on $K$, with $\{\Bder{f}{\bar x,z}:\ z\in K\}=
\{\alpha(z)\Bder{g}{\bar x}: z\in\R^n\}$.

(iii) Composition with differentiable mappings: if $f:\R^n\times\R^n\longrightarrow\R^p$
is $B$-differentiable at $\bar x$ uniformly on $K$ and $g:\R^p\longrightarrow\R^m$ is
Fr\'echet differentiable at each point $f(\bar x,z)$, with $z\in K$, then their
composition $g\circ f$ turns out to be $B$-differentiable at $\bar x$ uniformly on $K$,
with $\{\Bder{(g\circ f)}{\bar x,z}:\ z\in K\}=\{\nabla g(f(\bar x,z))\Bder{f}{\bar x,z}:
\ z\in K\}$.
\end{example}

A stronger notion of uniform $B$-differentiability will be needed for
one of the main results, which is based on strict $B$-differentiability.
Given a nonempty set $K\subseteq\R^n$, a bifunction $f:\R^n\times\R^n
\longrightarrow\R^m$ is said to be strictly $B$-differentiable at $\bar x\in K$,
uniformly on $K$, if there exists a family $\{\Bder{f}{\bar x,z}\in
\PH(\R^n,\R^m):\ z\in K\}$ such that for every $\epsilon>0$ $\exists\delta_\epsilon>0$
such that
$$
  \sup_{z\in K}{\|f(x_1,z)-f(x_2,z)-\Bder{f}{\bar x,z}(x_1-x_2)\|
  \over\|x_1-x_2\|}<\epsilon,\quad\forall x_1,\, x_2
  \in\ball{\bar x}{\delta_\epsilon},\ x_1\ne x_2.
$$

\subsection{Distance from strong vector equilibria}

The function $\nu:\R^n\longrightarrow [0,+\infty)$, defined by
\begin{equation}    \label{eq:defnumf}
  \nu(x)=\sup_{z\in K}\dist{f(x,z)}{C},
\end{equation}
can be exploited as a natural measure of the distance of a given
point $x\in\R^n$ from being a solution to $(\VEP)$.
Clearly it is $\Equi=\nu^{-1}(0)\cap K$, while positive values of $\nu$
quantify the violation of the strong equilibrium condition in $(\VEP)$.

A local error bound (in terms of $vu$) is said to be valid near $\bar x\in\Equi$ for
problem $(\VEP)$ if there exist positive $\kappa$ and $\delta$
such that
\begin{equation}     \label{in:erboSE}
   \dist{x}{\Equi}\le\kappa\nu(x),\quad\forall x\in
   \ball{\bar x}{\delta}\cap K.
\end{equation}
Notice that, whereas for computing $\dist{x}{\Equi}$ one needs to
know all the solutions to $(\VEP)$ near $\bar x$, the value of
$\nu(x)$ can be computed directly by means of problem data.
A study of sufficient conditions for the error in bound in
$(\ref{in:erboSE})$ to hold has been recently undertaken in \cite{Uder22}.
In particular, the following global error bound condition
under an uniform $B$-differentiability assumption on $f$ is known to
hold.

\begin{proposition}[\cite{Uder22}]
With reference to a problem $(\VEP)$, suppose that:
\begin{itemize}

\item[(i)] each function $x\mapsto f(x,z)$ is $C$-u.s.c.
on $K$, for every $z\in K$;

\item[(ii)] the set-valued mapping $x\leadsto f(x,K)$ takes $C$-bounded
values on $K$;

\item[(iii)] $K$ is convex;

\item[(iv)] $f$ is $B$-differentiable uniformly on $K$  at each
point of $K\backslash\Equi$;

\item[(v)] there exists $\sigma>0$ with the property that for every $x_0\in
K\backslash\Equi$ there is $u_0\in\Usfer\cap\cone(K-x_0)$ such that
$$
  \Bder{f}{x_0,z}(u_0)+\sigma\Uball\subseteq C,\quad\forall z\in K.
$$
\end{itemize}
Then, $\Equi$ is nonempty, closed and the following estimate holds true
$$
  \dist{x}{\Equi}\le{\nu(x)\over\sigma},\quad\forall x\in K.
$$
\end{proposition}

\vskip1cm


\section{Tangential approximation of $\Equi$}    \label{Sect:3}

\begin{theorem}[Inner approximation]     \label{thm:innerapprox}
With reference to a problem $(\VEP)$, let $\bar x\in\Equi$.
Suppose that:

\begin{itemize}

\item[(i)] $f$ is $B$-differentiable at $\bar x$, uniformly on $K$,
with $\{\Bder{f}{\bar x,z}:\ z\in K\}$;

\item[(ii)] a local error bound such as $(\ref{in:erboSE})$
is valid near $\bar x$.

\end{itemize}
Then, it holds
\begin{equation}    \label{in:inapproxTangEqui}
  \bigcap_{z\in K}\Bder{f}{\bar x,z}^{-1}(C)\cap\Wang{\bar x}{K}
  \subseteq\Tang{\bar x}{\Equi}.
\end{equation}
\end{theorem}

\begin{proof}
Let us start with observing that, since it is $\Bder{f}{\bar x,z}\in
\PH(\R^n,\R^m)$ for every $z\in K$, and $C$ is a cone, each set
$\Bder{f}{\bar x,z}^{-1}(C)$ turns out to be a cone containing $\nullv$,
as well as $\Wang{\bar x}{K}$ does by definition.
Thus, if taking $v=\nullv\in\bigcap_{z\in K}\Bder{f}{\bar x,z}^{-1}(C)
\cap\Wang{\bar x}{K}$, the inclusion $v\in\Tang{\bar x}{\Equi}$ obviously
holds as the latter cone is closed.
So, take an arbitrary $v\in\left(\bigcap_{z\in K}\Bder{f}{\bar x,z}^{-1}(C)
\cap\Wang{\bar x}{K}\right)\backslash\{\nullv\}$.
Since both the sets in the inclusion in $(\ref{in:inapproxTangEqui})$
are cones, one can assume without any loss of generality that $\|v\|=1$.
In the light of the characterization via $(\ref{eq:charTangcone})$,
$v$ is proven to belong to $\Tang{\bar x}{\Equi}$ if one shows that
\begin{equation}      \label{eq:thesisreform}
    \liminf_{t\downarrow 0}{\dist{\bar x+tv}{\Equi}\over t}=0.
\end{equation}
Showing the equality in $(\ref{eq:thesisreform})$ amounts to show
that for every $\tau>0$ and $\epsilon>0$ there exists $t_0\in (0,\tau)$
such that
\begin{equation}    \label{eq:thesisreform2}
   {\dist{\bar x+t_0v}{\Equi}\over t_0}\le\epsilon.
\end{equation}
So, let us fix ad libitum $\tau$ and $\epsilon$.
Hypothesis (ii) ensures the existence of $\delta,\ \kappa>0$ as
in $(\ref{in:erboSE})$.
By virtue of hypothesis (i), corresponding to $\epsilon/\kappa$,
there exists $\delta_\epsilon>0$ such that
$$
  f(x,z)\in f(\bar x,z)+\Bder{f}{\bar x,z}(x-\bar x)+\kappa^{-1}\epsilon
  \|x-\bar x\|\Uball, \quad\forall x\in\ball{\bar x}{\delta_\epsilon},
  \ \forall z\in K,
$$
and hence, in particular,
$$
  f(\bar x+tv,z)\in f(\bar x,z)+t\Bder{f}{\bar x,z}(v)+\kappa^{-1}\epsilon
  t\Uball, \quad\forall t\in (0,\delta_\epsilon),
  \ \forall z\in K.
$$
By taking into account that $\bar x\in\Equi$ and $v\in\Bder{f}{\bar x,z}^{-1}(C)$
for every $z\in K$, the above inclusion implies
$$
   f(\bar x+tv,z)\in C+tC+\kappa^{-1}\epsilon t\Uball\subseteq
   C+\kappa^{-1}\epsilon t\Uball,\quad\forall t\in (0,\delta_\epsilon),
  \ \forall z\in K.
$$
In terms of the residual function $\nu$ introduced in $(\ref{eq:defnumf})$,
this means
\begin{eqnarray}     \label{in:resepst}
  \nu(\bar x+tv)=\sup_{z\in K}\dist{f(\bar x+tv,z)}{C} &\le&
  \exc{C+\kappa^{-1}\epsilon t\Uball}{C}=
  \exc{\kappa^{-1}\epsilon t\Uball}{C}  \nonumber  \\
  &\le& \kappa^{-1}\epsilon t ,\quad\forall t\in (0,\delta_\epsilon),
\end{eqnarray}
where the second equality holds because $C$ is a convex cone.
On the other hand, according to hypothesis (ii) there exists $\delta_0
\in (0,\min\{\tau,\delta,\delta_\epsilon\})$ such that
\begin{equation}    \label{in:erboEquidelta}
   \dist{x}{\Equi}\le\kappa\nu(x),\quad\forall x\in\ball{\bar x}{\delta_0}
   \cap K.
\end{equation}
Since it is $v\in\Wang{\bar x}{K}$, for some $t_*\in (0,\delta_0)$ it happens
$$
   \bar x+t_*v\in K\cap\ball{\bar x}{\delta_0},
$$
and therefore, by inequality $(\ref{in:erboEquidelta})$, one obtains
\begin{eqnarray}    \label{in:distnuv}
   \dist{\bar x+t_* v}{\Equi}\le\kappa\nu(\bar x+t_* v).
\end{eqnarray}
By combining inequalities $(\ref{in:resepst})$ and $(\ref{in:distnuv})$,
as it is $t_*<\delta_0<\delta_\epsilon$, one obtains
$$
    \dist{\bar x+t_* v}{\Equi}\le\kappa\cdot\kappa^{-1}\epsilon t_*
    =\epsilon t_*.
$$
The last inequality shows that $(\ref{eq:thesisreform2})$ is true
for $t_0=t_*\in (0,\tau)$, thereby completing the proof.
\end{proof}

The inclusion in $(\ref{in:inapproxTangEqui})$ states that, under proper
assumptions, any solution of the (approximated) problem
\begin{equation}    \label{in:HomVEP}
  \hbox{ find $v\in\Wang{\bar x}{K}$ such that } \Bder{f}{\bar x;z}(v)
  \in C,  \quad\forall z\in K,
\end{equation}
provides a vector, which is tangent to $\Equi$ at $\bar x$ in the sense
of Bouligand. Notice that problem $(\ref{in:HomVEP})$ is almost in the form
$(\VEP)$ (it would be exactly in the form $(\VEP)$ if $\Wang{\bar x}{K}=K$).
Roughly speaking, all of this means that if the problem data of
$(\VEP)$ are properly approximated ($K$ by its radial direction cone,
$f$ by its generalized derivatives in the sense of Bouligand,
respectively) near a reference solution $\bar x$, then the solutions
of the resulting approximated problem $(\ref{in:HomVEP})$ work as
a first-order approximation of the solution set to the original problem
$(\VEP)$.
Problem $(\ref{in:HomVEP})$ is typically expected to be easier than
$(\VEP)$ by virtue of the structural properties of its data.
Basically, $(\ref{in:HomVEP})$ can be regarded as a cone constrained
p.h. vector inequality system, so its solution set is a cone.
Furthermore, if $K$ is convex and $\Bder{f}{\bar x,z}:\R^n
\longrightarrow\R^m$ is $C$-concave for every $z\in K$, the latter
meaning that
$$
  \Bder{f}{\bar x,z}(v_1)+\Bder{f}{\bar x,z}(v_2)\parord
  \Bder{f}{\bar x,z}(v_1+v_2),\quad\forall v_1,\, v_2\in\R^n,
$$
where $\parord$ denotes the partial ordering on $\R^m$ induced in the
standard way by the cone $C$,
then the solution set to problem $(\ref{in:HomVEP})$ is a convex cone.

As a further comment to Theorem \ref{thm:innerapprox}, it must be
remarked that the inclusion in $(\ref{in:inapproxTangEqui})$
provides only a one-side approximation of $\Tang{\bar x}{\Equi}$,
which may happen to be rather rough. This fact is illustrated by the
next example.

\begin{example}[Inclusion $(\ref{in:inapproxTangEqui})$ may be strict]
Consider the problem $(\VEP)$ defined by the following data: $K=C=\R^2_+=
\{x=(x_1,x_2)\in\R^2:\ x_1\ge 0,\ x_2\ge 0\}$ and a vector-valued bifunction
$f:\R^2\times\R^2\longrightarrow\R^2$ given by
$$
  f(x_1,x_2,z_1,z_2)=\left(\begin{array}{c} {1\over 2}(-m_z^-x_1+x_2+1)^2 \\
                            \\
                     {1\over 2}(m_z^+x_1-x_2+1)^2
                     \end{array}\right),
$$
where
$$
  m_z^-=1-{1\over \|z\|^2+1}\qquad \hbox{ and }\qquad
  m_z^+=1+{1\over \|z\|^2+1}, \quad\ z\in\R^2.
$$
Since $f(x,z)\in\R^2_+$ for every $(x,z)\in\R^2\times\R^2$, it is clear
that $\Equi=K=\R^2_+$. Fix $\bar x=\nullv\in\Equi$, so one has
$$
  \Wang{\nullv}{K}=\Tang{\nullv}{\Equi}=\R^2_+.
$$
In view of the next calculations, it is convenient to observe that
$$
  f(x,z)=(g\circ h)(x,z),
$$
where the mappings $g:\R^2\longrightarrow\R^2$ and $h:\R^2\times\R^2\longrightarrow\R^2$
are given respectively by
$$
  g(y)=\left(\begin{array}{c} y_1^2/2 \\
                              y_2^2/2
                     \end{array}\right)
  \qquad \hbox{ and }\qquad
  h(x,z)=\left(\begin{array}{c} -m_z^-x_1+x_2+1 \\
                               m_z^+x_1-x_2+1
                     \end{array}\right).
$$
To check that the bifunction $h$ is $B$-differentiable at $\nullv$
uniformly on $\R^2_+$, with
$$
  \left\{\Bder{h}{\nullv,z}=\nabla h(\nullv,z)=
  \left(\begin{array}{rr} -m_z^- & 1 \\
                            m_z^+ & -1
                     \end{array}\right),\ z\in\R^2_+
  \right\}
$$
it suffices to observe that
\begin{eqnarray*}
  \|h(x,z)-h(\nullv,z)-\Bder{h}{\nullv,z}(x)\| &=&
  \left\|
  \left(\begin{array}{c} -m_z^-x_1+x_2+1 \\
                               m_z^+x_1-x_2+1
                     \end{array}\right)-
                     \left(\begin{array}{c} 1 \\
                                            1
                     \end{array}\right)-
    \left(\begin{array}{rr} -m_z^- & 1 \\
                            m_z^+ & -1
                     \end{array}\right)
       \left(\begin{array}{c} x_1 \\
                              x_2
                     \end{array}\right)
  \right\|  \\
  &=& 0, \quad\forall z\in\R^2_+.
\end{eqnarray*}
Thus, since $g$ is Fr\'echet differentiable at each point of $\R^2$
and
$$
  \nabla g(y)=\left(\begin{array}{rr} y_1 & 0 \\
                                      0  & y_2
                     \end{array}\right),
$$
according to what remarked in Example \ref{ex:unifBdiff3}(iii),
the mapping $f=g\circ h$ turns out to be $B$-differentiable at
$\nullv$ uniformly on $\R^2_+$, with
$$
 \Bder{f}{\nullv,z}=\nabla g(h(\nullv,z))\circ \Bder{h}{\nullv,z}=
 \left(\begin{array}{rr} 1 & 0 \\
                         0  & 1
                     \end{array}\right)
  \left(\begin{array}{rr} -m_z^- & 1 \\
                            m_z^+ & -1
                     \end{array}\right)=
                     \left(\begin{array}{rr} -m_z^- & 1 \\
                            m_z^+ & -1
                     \end{array}\right),\ z\in\R^2_+.
$$
Notice that a local error bound as in $(\ref{in:erboSE})$ is
evidently valid near $\nullv$ because it is $\Equi=K$.
Thus, all the hypotheses of Theorem \ref{thm:innerapprox} are
satisfied.

Now, one readily sees that
$$
  \Bder{f}{\nullv,z}(v)=\left(\begin{array}{c} -m_z^-v_1+v_2 \\
                            m_z^+v_1-v_2
                     \end{array}\right)\in\R^2_+
                     \qquad\hbox{ iff }\qquad
                     \left\{\begin{array}{c} -m_z^-v_1+v_2\ge 0 \\
                     \\
                            m_z^+v_1-v_2\ge 0.
                     \end{array}\right.
$$
This leads to find
$$
  \Bder{f}{\nullv,z}^{-1}(\R^2_+)=\{v\in\R^2:\
  m_z^-v_1\le v_2\le m_z^+v_1\},\quad\forall z\in\R^2_+.
$$
Since one has
$$
  \lim_{\|z\|\to\infty} m_z^-=1^-=1=1^+=
  \lim_{\|z\|\to\infty} m_z^+,
$$
it results in
$$
  \bigcap_{z\in\R^2_+}\Bder{f}{\nullv,z}^{-1}(\R^2_+)\cap
  \Wang{\nullv}{\R^2_+}=
  \{v\in\R^2_+:\ v_2=v_1\}\subsetneqq \R^2_+=\Tang{\nullv}{\Equi}.
$$
\end{example}

The above example motivates the interest in outer approximations
of $\Equi$. Below, a result in this direction is presented.

\begin{theorem}[Outer approximation]   \label{thm:outapprox}
With reference to a problem $(\VEP)$, let $\bar x\in\Equi$.
Suppose that:

\begin{itemize}

\item[(i)] $f$ is strictly $B$-differentiable at $\bar x$, uniformly
on $K$, with $\{\Bder{f}{\bar x,z}:\ z\in K\}$;

\item[(ii)] the family of mappings $\{\Bder{f}{\bar x,z}:\ z\in K\}$ is
equicontinuous at each point of $\R^n$.

\end{itemize}
Then, it holds
\begin{equation}    \label{in:outapproxTangEqui}
  \Tang{\bar x}{\Equi}\subseteq\bigcap_{z\in K}
  \Bder{f}{\bar x,z}^{-1}(\Tang{f(\bar x,z)}{C})
  \cap\Tang{\bar x}{K}.
\end{equation}
\end{theorem}

\begin{proof}
Since it is $\Bder{f}{\bar x,z}\in\PH(\R^n,\R^m)$ for every $z\in K$,
one has
$$
  \Bder{f}{\bar x,z}(\nullv)=\nullv\in\Tang{f(\bar x,z)}{C},
  \quad\forall z\in K.
$$
Therefore, it clearly holds
$$
  \nullv\in\bigcap_{z\in K}
  \Bder{f}{\bar x,z}^{-1}(\Tang{f(\bar x,z)}{C})
  \cap\Tang{\bar x}{K}.
$$
So take an arbitrary $v\in\Tang{\bar x}{\Equi}\backslash\{\nullv\}$.
As all the sets involved in inclusion $(\ref{in:outapproxTangEqui})$
are cones, without loss
of generality it is possible to assume that $\|v\|=1$. According
to the definition of contingent cone, there exist $(v_n)_n$, with
$v_n\longrightarrow v$ and $(t_n)_n$, with $t_n\downarrow 0$,
such that $\bar x+t_nv_n\in\Equi\subseteq K$. Notice that this inclusion
in particular implies that $v\in\Tang{\bar x}{K}$. What remains to
be shown is that
\begin{equation}    \label{in:thesreformoutapprox}
   v\in\bigcap_{z\in K}\Bder{f}{\bar x,z}^{-1}(\Tang{f(\bar x,z)}{C}).
\end{equation}
Fix an arbitrary $\epsilon>0$. By virtue of hypothesis (i), there
exists $\delta_\epsilon>0$ such that
$$
  f(x_1,z)-f(x_2,z)-\Bder{f}{\bar x,z}(x_1-x_2)\in\epsilon
  \|x_1-x_2\|\Uball,\quad\forall z\in K,\ \forall x_1,\, x_2\in
  \ball{\bar x}{\delta_\epsilon}
$$
and hence
\begin{equation}    \label{in:fstrBdifx1x2}
   \Bder{f}{\bar x,z}(x_1-x_2)\in f(x_1,z)-f(x_2,z)+\epsilon
  \|x_1-x_2\|\Uball,\quad\forall z\in K,\ \forall x_1,\, x_2\in
  \ball{\bar x}{\delta_\epsilon}.
\end{equation}
Since it is $\bar x+t_nv_n\longrightarrow\bar x$ as $n\to\infty$
(as a converging sequence $(v_n)_n$ must be bounded), for some $n_\epsilon\in\N$
it is true that $\bar x+t_nv_n\in\ball{\bar x}{\delta_\epsilon}$
for every $n\ge n_\epsilon$.
Thus, by taking $x_1=\bar x+t_nv_n$ and $x_2=\bar x$ in $(\ref{in:fstrBdifx1x2})$,
one finds
$$
   t_n\Bder{f}{\bar x,z}(v_n)\in f(\bar x+t_nv_n,z)-f(\bar x,z)+
   \epsilon t_n\|v_n\|\Uball,\quad\forall z\in K,\ \forall n\ge n_\epsilon,
$$
whence it follows
$$
   \Bder{f}{\bar x,z}(v_n)\in {f(\bar x+t_nv_n,z)-f(\bar x,z)\over t_n}
   +\epsilon\|v_n\|\Uball,\quad\forall z\in K,\ \forall n\ge n_\epsilon.
$$
By taking into account that $v_n\longrightarrow v$ as $n\to\infty$ and $\|v\|=1$,
one has that $\|v_n\|\le 2$ for all $n\ge n_\epsilon$, up to a proper increase
in the value of $n_\epsilon$, if needed. Thus, from the last inclusion
one obtains
\begin{equation}    \label{in:Bdervnincrrepeps}
   \Bder{f}{\bar x,z}(v_n)\in {f(\bar x+t_nv_n,z)-f(\bar x,z)\over t_n}
   +2\epsilon\Uball,\quad\forall z\in K,\ \forall n\ge n_\epsilon.
\end{equation}
By hypothesis (ii) the family $\{\Bder{f}{\bar x,z}:\ z\in K\}$ is equicontinuous
at $v$. This means that there exists $n_*\in\N$ (independent of $z$),
with $n_*\ge n_\epsilon$, such that
$$
   \|\Bder{f}{\bar x,z}(v_n)-\Bder{f}{\bar x,z}(v)\|\le\epsilon,
   \quad\forall z\in K,\ \forall n\ge n_*,
$$
or, equivalently,
$$
  \Bder{f}{\bar x,z}(v)\in\Bder{f}{\bar x,z}(v_n)+\epsilon\Uball,
  \quad\forall z\in K,\ \forall n\ge n_*.
$$
By recalling $(\ref{in:Bdervnincrrepeps})$, from the last inclusion
one gets
$$
  \Bder{f}{\bar x,z}(v)\in {f(\bar x+t_nv_n,z)-f(\bar x,z)\over t_n}
   +3\epsilon\Uball,\quad\forall z\in K,\ \forall n\ge n_*.
$$
Since it is $\bar x+t_nv_n\in\Equi$ for every $n\in\N$, this implies
$$
  \Bder{f}{\bar x,z}(v)\in {C-f(\bar x,z)\over t_n}
   +3\epsilon\Uball\in \cone(C-f(\bar x,z))+3\epsilon\Uball,
   \quad\forall z\in K,\ \forall n\ge n_*.
$$
Since $C$ is convex so $\Tang{f(\bar x,z)}{C}=\cl\cone(C-f(\bar x,z)))$,
it results in
$$
  \Bder{f}{\bar x,z}(v)\in \Tang{f(\bar x,z)}{C}
   +3\epsilon\Uball,\quad\forall z\in K.
$$
The arbitrariness of $\epsilon$ and the fact $\Tang{f(\bar x,z)}{C}$ is
closed allow one to assert that
$$
  \Bder{f}{\bar x,z}(v)\in \Tang{f(\bar x,z)}{C},\quad\forall z\in K,
$$
which proves the validity of $(\ref{in:thesreformoutapprox})$.
Thus the proof is complete.
\end{proof}

\begin{remark}     \label{rem:refthmoutapprox}
(i) In the case in which $\inte C\ne\varnothing$, it is useful to remark
that the formula in $(\ref{in:outapproxTangEqui})$ can be equivalently rewritten
as
$$
  \Tang{\bar x}{\Equi}\subseteq \{\nullv\}\cup
  \left(\bigcap_{z\in K\cap f^{-1}(\bar x,\cdot)(\bd C)}
  \Bder{f}{\bar x,z}^{-1}(\Tang{f(\bar x,z)}{C})
  \cap\Tang{\bar x}{K}\right),
$$
with the convention that an intersection over an empty index set is the empty set.
Indeed, whenever it happens $f(\bar x,z)\in\inte C$, one has $\Tang{f(\bar x,z)}{C}=\R^m$,
with the consequence that $\Bder{f}{\bar x,z}^{-1}(\Tang{f(\bar x,z)}{C})=
\R^n$.

(ii) It is worth noticing that for all those $z_0\in K$ such that
$f(\bar x,z_0)=\nullv$ (if any), the formula in $(\ref{in:outapproxTangEqui})$
entails
$$
  \Tang{\bar x}{\Equi}\subseteq\Bder{f}{\bar x,z_0}^{-1}(C)\cap\Tang{\bar x}{K},
$$
as it is $\Tang{f(\bar x,z_0)}{C}=\Tang{\nullv}{C}=C$.
\end{remark}

The next example shows that also the outer approximation of $\Tang{\bar x}{\Equi}$
provided by Theorem \ref{thm:outapprox} may happen to be rather rough.

\begin{example}[Inclusion $(\ref{in:outapproxTangEqui})$ may be strict]
Consider the (actually scalar) problem $(\VEP)$ defined by the following data:
$K=\R$, $C=[0,+\infty)$, $f:\R\times\R\longrightarrow\R$  given by
$$
  f(x,z)={x^2z\over z^2+1}.
$$
It is clear that $\Equi=\{0\}$. So, fix $\bar x=0$.
In order for checking that $f$ is strictly $B$-differentiable at $0$ uniformly
on $\R$, with $\{\Bder{f}{0,z}\equiv 0,\ z\in\R\}$, according to the
definition it suffices to observe that, fixed an arbitrary $\epsilon>0$,
one has
\begin{eqnarray*}
  \sup_{z\in\R}{|f(x_1,z)-f(x_2,z)|\over |x_1-x_2|} &=&
  \sup_{z\in\R}{\displaystyle{\left|{x_1^2z\over z^2+1}-{x_2^2z\over z^2+1}\right|}\over |x_1-x_2|} =
  \sup_{z\in\R}{|z|\over z^2+1}\cdot |x_1+x_2|\le |x_1|+|x_2| \\
  &\le & \epsilon,\quad\forall x_1,\, x_2\in\ball{0}{\epsilon/2},\ x_1\ne x_2.
\end{eqnarray*}
As the family $\{\Bder{f}{0,z}\equiv 0,\ z\in\R\}$ is actually independent of $z\in\R$,
also hypothesis (ii) of Theorem \ref{thm:outapprox} is satisfied.

Since $f(0,z)=0$ for every $z\in\R$, so it is $\Tang{f(0,z)}{[0,+\infty)}=[0,+\infty)$,
one finds
$$
  \Bder{f}{0,z}^{-1}\left(\Tang{f(0,z)}{[0,+\infty)}\right)=\R,\quad\forall z\in\R.
$$
Consequently, in the current case, one obtains
$$
  \Tang{0}{\Equi}=\{0\}\subsetneqq\R\cap\R=\bigcap_{z\in\R}
  \Bder{f}{0,z}^{-1}(\Tang{f(0,z)}{[0,+\infty)})
  \cap\Tang{0}{\R}.
$$
\end{example}

Relying on both the preceding approximations, the next result singles
out a sufficient condition, upon which one can establish an exact
representation of $\Tang{\bar x}{\Equi}$.

\begin{corollary}
With reference to a problem $(\VEP)$, let $\bar x\in\Equi$.
Suppose that:

\begin{itemize}

\item[(i)] $K$ is polyhedral;

\item[(ii)] $f(\bar x,z)=\nullv,\quad\forall z\in K$;

\item[(iii)] $f$ is strictly $B$-differentiable at $\bar x$, uniformly
on $K$, with $\{\Bder{f}{\bar x,z}:\ z\in K\}$;

\item[(iv)] the family of mappings $\{\Bder{f}{\bar x,z}:\ z\in K\}$ is
equicontinuous at each point of $\R^n$;

\item[(v)] a local error bound such as in $(\ref{in:erboSE})$ is valid near $\bar x$.

\end{itemize}
Then, it holds
$$
  \Tang{\bar x}{\Equi}=\bigcap_{z\in K}\Bder{f}{\bar x,z}^{-1}(C)
  \cap\Tang{\bar x}{K}.
$$
\end{corollary}

\begin{proof}
The above assumptions enable one to apply both Theorem \ref{thm:innerapprox}
and Theorem \ref{thm:outapprox}.
From the former one, in the light of Remark \ref{rem:polyWang} and hypothesis (i),
one obtains
\begin{equation}    \label{in:innerTang}
 \bigcap_{z\in K}\Bder{f}{\bar x,z}^{-1}(C)\cap\Tang{\bar x}{K}\subseteq
 \Tang{\bar x}{\Equi}.
\end{equation}
From the latter, in the light of hypothesis (ii) and Remark \ref{rem:refthmoutapprox}(ii),
one obtains
\begin{equation} \label{in:outTang}
  \Tang{\bar x}{\Equi}\subseteq\bigcap_{z\in K}\Bder{f}{\bar x,z}^{-1}(C)
  \cap\Tang{\bar x}{K}.
\end{equation}
By combining inclusions $(\ref{in:innerTang})$ and $(\ref{in:outTang})$ one gets
the equality in the thesis.
\end{proof}

\vskip1cm


\section{Applications to constrained optimization}    \label{Sect:4}

This section deals with first-order optimality conditions for
optimization problems, whose feasible region is formalized as a set of
strong vector equilibria. As such, these problems
can be cast in mathematical programming with equilibrium constraints,
a well-recognized topic and active area of research (see, among others,
\cite{LuPaRa96,LuPaRaWu96,Mord06b,OuKoZo98,Ye05}).
Thus, the optimization problems here considered take the following
form
$$
  \min\vartheta(x) \quad\hbox{ subject to }\quad x\in\Equi,
  \leqno (\MPVEC)
$$
where $\vartheta:\R^n\longrightarrow\R$ is the objective function formalizing
the criterion used for comparing variables,
while $\Equi$ is the feasible region of the problem, denoting as in
the previous sections the solution sets to an inner problem $(\VEP)$.
Throughout this section $\vartheta$ will be assumed to be continuous
around $\bar x$, but possibly nondifferentiable, as well as the bifunction
$f$ defining $(\VEP)$.

In constrained nondifferentiable optimization, first-order optimality conditions are
typically obtained by locally approximating the objective function and the feasible region of
a given problem. In this vein, the fact stated in the next lemma is widely
known to hold, which has been used as a starting point for various, more
elaborated, optimality conditions. For a direct proof see, for instance, \cite[Chapter 7.1]{Schi07}.
To a deeper view, it can be restored as a special case of an
axiomatic scheme of analysis, which was developed in \cite{CasPap95,ElsThi88}
(see \cite[Theorem 2.1]{CasPap95}).

\begin{lemma}
Let $\bar x\in\Equi$ be a local optimal solution to problem $(\MPVEC)$.
Then, it holds
\begin{equation}    \label{in:nocDTr}
   \Duder{\vartheta}{\bar x}{w}\ge 0,\quad\forall w\in\Wang{\bar x}{\Equi}
\end{equation}
and
\begin{equation}     \label{in:nocDHT}
   \DHuder{\vartheta}{\bar x}{w}\ge 0,\quad\forall w\in\Tang{\bar x}{\Equi}.
\end{equation}
\end{lemma}

\begin{remark}
Since from their very definition one sees that
$$
  \Duder{\vartheta}{\bar x}{w}\le\DHuder{\vartheta}{\bar x}{w},
  \quad\forall w\in\R^n,
$$
whereas it is $\Wang{\bar x}{\Equi}\subseteq\Tang{\bar x}{\Equi}$,
none of the conditions $(\ref{in:nocDTr})$ and $(\ref{in:nocDHT})$ can imply
in general the other, unless $\vartheta$ is locally Lipschitz near $\bar x$
or it is $\Wang{\bar x}{\Equi}=\Tang{\bar x}{\Equi}$. Thus, the author does
not agree with what asserted in \cite[pag. 132]{Schi07}.
For the purposes of the present analysis, only the
condition in $(\ref{in:nocDHT})$  will be actually exploited.
\end{remark}

\begin{theorem}[Necessary optimality condition]    \label{thm:NOCMPVEC}
Let $\bar x\in\Equi$ be a local optimal solution to problem $(\MPVEC)$.
Suppose that:
\begin{itemize}

\item[(i)] $f$ is $B$-differentiable at $\bar x$, uniformly on $K$,
with $\{\Bder{f}{\bar x,z}:\ z\in K\}$;

\item[(ii)] a local error bound such as in $(\ref{in:erboSE})$
is valid near $\bar x$.

\end{itemize}
Then, it holds
\begin{equation}    \label{in:NOCMPVEC}
   -\Upsubd\vartheta(\bar x)\subseteq \dcone{\left(
   \bigcap_{z\in K}\Bder{f}{\bar x,z}^{-1}(C)\cap\Wang{\bar x}{K}\right)} .
\end{equation}
\end{theorem}

\begin{proof}
Under the above assumptions, by Theorem \ref{thm:innerapprox} the inclusion
in $(\ref{in:inapproxTangEqui})$ holds true. Consequently, since
$\bar x\in\Equi$ is a local optimal solution to $(\MPVEC)$,
according to condition $(\ref{in:nocDHT})$ it must be
$$
  \DHuder{\vartheta}{\bar x}{w}\ge 0,\quad\forall w\in
  \bigcap_{z\in K}\Bder{f}{\bar x,z}^{-1}(C)\cap\Wang{\bar x}{K}.
$$
If $\Upsubd\vartheta(\bar x)=\varnothing$ the thesis becomes trivial.
Otherwise, by taking into account the representation in $(\ref{eq:UpsubdDHder})$,
which is valid because the function $\vartheta$ is in particular
u.s.c. around $\bar x$, for an arbitrary $v\in\Upsubd\vartheta(\bar x)$
one finds
$$
  \langle v,w\rangle\ge 0,\quad\forall w\in
  \bigcap_{z\in K}\Bder{f}{\bar x,z}^{-1}(C)\cap\Wang{\bar x}{K},
$$
which amounts to say that
$$
  -v\in\dcone{\left(\bigcap_{z\in K}\Bder{f}{\bar x,z}^{-1}(C)
  \cap\Wang{\bar x}{K}\right)}.
$$
The arbitrariness of $v\in\Upsubd\vartheta(\bar x)$ completes
the proof.
\end{proof}

\begin{remark}
To assess the role of the optimality condition formulated in Theorem \ref{thm:NOCMPVEC},
notice that it does not carry useful information whenever $\Fsubd\vartheta(\bar x)=\varnothing$.
This happens, for example, if $\vartheta$ is a convex continuous function,
which is nondifferentiable at $\bar x$. Nevertheless, the upper subdifferential
is nonempty for large classes of functions, including the class of semiconcave
ones (see \cite{Mord06b}). In all such cases, condition $(\ref{in:NOCMPVEC})$
provides a necessary optimality condition, which may be more efficient than
those expressed in terms of more traditional lower subdifferentials. This because it requires
that all elements in $ -\Upsubd\vartheta(\bar x)$ belong to the set in the right-side
of $(\ref{in:NOCMPVEC})$, in contrast to a mere nonempty intersection requirement,
which is typical for the lower subdifferential case.
\end{remark}

\begin{corollary}
Under the same assumptions of Theorem \ref{thm:NOCMPVEC}, if the following
additional hypotheses are satisfied:

\begin{itemize}

\item[(i)] $K$ is polyhedral;

\item[(ii)] $\Bder{f}{\bar x,z}\in\PH(\R^n,\R^m)$ is $C$-concave
for every $z\in K$;

\item[(iii)] the qualification condition holds
\begin{equation}   \label{in:qcnocMPVEC}
  \bigcap_{z\in K}\Bder{f}{\bar x,z}^{-1}(C)
  \cap
  \inte\Tang{\bar x}{K}\ne\varnothing,
\end{equation}

\end{itemize}
then the inclusion in $(\ref{in:NOCMPVEC})$ takes the
simpler form
$$
   -\Upsubd\vartheta(\bar x)\subseteq \dcone{\left(
   \bigcap_{z\in K}\Bder{f}{\bar x,z}^{-1}(C)\right)}
   +\Ncone{\bar x}{K}.
$$
\end{corollary}

\begin{proof}
It is well know that if $S_1$ and $S_2$ are closed convex cones,
then $\dcone{(S_1\cap S_2)}=\cl(\dcone{S_1}+\dcone{S_2})$ (see
\cite[Lemma 2.4.1]{Schi07}). On the other hand, if $S_1-S_2=\R^n$,
then $\dcone{S_1}+\dcone{S_2}$ is closed (see \cite[Proposition 2.4.3]{Schi07}
If the qualification condition $S_1\cap\inte S_2\ne\varnothing$ happens
to be satisfied, then $S_1-S_2=\R^n$ (see \cite[Lemma 2.4.4]{Schi07}). Thus, since
$\bigcap_{z\in K}\Bder{f}{\bar x,z}^{-1}(C)$ and $\Tang{\bar x}{K}$
are closed convex cone, by virtue of $(\ref{in:qcnocMPVEC})$ and
the assumption (i), one obtains
$$
  \dcone{\left(\bigcap_{z\in K}\Bder{f}{\bar x,z}^{-1}(C)
  \cap\Wang{\bar x}{K}\right)}=\dcone{\left(\bigcap_{z\in K}\Bder{f}{\bar x,z}^{-1}(C)\right)}
  +\dcone{\Tang{\bar x}{K}}.
$$
Then, in order to achieve the inclusion in the thesis
it suffices to recall that $\dcone{\Tang{\bar x}{K}}=
\Ncone{\bar x}{K}$ (see \cite[Lemma 11.2.2]{Schi07}).
\end{proof}

Now, let us consider sufficient optimality conditions, a topic usually
investigated in a subsequent step of analysis.

The next lemma provides a sufficient optimality condition for
$(\MPVEC)$ in the case the objective function is locally Lipschitz.
For its proof see \cite[Lemma 1.3, Chapter V]{DemRub95}. Notice that for
the statement of Lemma \ref{lem:SOCMPVEC}, the hypothesis
on the feasible region of the problem to allow a first-order uniform
conical approximation in the sense of Demyanov-Rubinov is not needed
(see \cite[Remark 1.6, Chapter V]{DemRub95}).

\begin{lemma}    \label{lem:SOCMPVEC}
With reference to $(\MPVEC)$, suppose that $\vartheta$ is locally
Lipschitz around $\bar x\in\Equi$. If it holds
\begin{equation}     \label{in:socDlolem}
  \Dlder{\vartheta}{\bar x}{w}>0,\quad\forall w\in
  \Tang{\bar x}{\Equi}\backslash\{\nullv\},
\end{equation}
then $\bar x$ is a strict local solution to $(\MPVEC)$.
\end{lemma}

On the base of the above lemma, one is in a position to establish the next result.

\begin{theorem}[Sufficient optimality condition]     \label{thm:SOCMPVEC}
With reference to $(\MPVEC)$, assume that $\vartheta$ is locally
Lipschitz around $\bar x\in\Equi$. Suppose that:

\begin{itemize}

\item[(i)] $f$ is strictly $B$-differentiable at $\bar x$, uniformly
on $K$, with $\{\Bder{f}{\bar x,z}:\ z\in K\}$;

\item[(ii)] the family of mappings $\{\Bder{f}{\bar x,z}:\ z\in K\}$ is
equicontinuous at each point of $\R^n$.

\end{itemize}
If the condition
\begin{equation}     \label{in:socDlo}
  \nullv\in\Fsubd\vartheta(\bar x)+\inte\left[\dcone{\left(\bigcap_{z\in K}
  \Bder{f}{\bar x,z}^{-1}(\Tang{f(\bar x,z)}{C})\cap\Tang{\bar x}{K}\right)}
  \right],
\end{equation}
is satisfied, then $\bar x$ is a strict local solution to $(\MPVEC)$.
\end{theorem}

\begin{proof}
Observe first that if for a given cone $S\subseteq\R^n$ it is
$v\in\inte(\dcone{S})$, then it must be
$$
  \langle v,s\rangle<0,\quad\forall s\in S\backslash\{\nullv\}.
$$
Indeed, there exists $\delta>0$ such that $v+\delta\Uball\subseteq
\dcone{S}$, and therefore it holds
$$
  \langle v+\delta u,s\rangle\le 0,\quad\forall u\in\Uball,
  \ \forall s\in S.
$$
Thus, for any $s\in S\backslash\{\nullv\}$, the last inequality
implies
$$
  \sup_{u\in\Uball}\langle v+\delta u,s\rangle=
  \langle v,s\rangle+\delta\sup_{u\in\Uball}\langle u,s\rangle=
  \langle v,s\rangle+\delta\|s\|\le 0,
$$
whence one gets
$$
  \langle v,s\rangle\le -\delta\|s\|<0.
$$
Consequently, the condition $(\ref{in:socDlo})$ implies that there exists
$v\in\Fsubd\vartheta(\bar x)$ such that
it is
$$
  \langle v,w\rangle>0,\quad\forall w\in\left[\bigcap_{z\in K}
  \Bder{f}{\bar x,z}^{-1}(\Tang{f(\bar x,z)}{C})\cap\Tang{\bar x}{K}\right]
  \backslash\{\nullv\}.
$$
By recalling the representation of $\Fsubd\vartheta(\bar x)$ in
$(\ref{eq:FsubdDHder})$, from the last inequality one obtains
$$
  \Dlder{\vartheta}{\bar x}{w}=\DHlder{\vartheta}{\bar x}{w}>0,
  \quad\forall w\in\left[\bigcap_{z\in K}
  \Bder{f}{\bar x,z}^{-1}(\Tang{f(\bar x,z)}{C})\cap\Tang{\bar x}{K}\right]
  \backslash\{\nullv\}.
$$
Since under the above assumptions Theorem \ref{thm:outapprox}
can be applied, then by virtue of the inclusion in $(\ref{in:outapproxTangEqui})$
one can state that condition $(\ref{in:socDlolem})$ turns out to be
satisfied. Thus, the thesis of the theorem follows from Lemma
\ref{lem:SOCMPVEC}.
\end{proof}

\begin{remark}
(i) As it is possible to see by elementary examples (see \cite[Chapter 1]{Mord18}),
$\Fsubd\vartheta(\bar x)$ may happen to be empty even though $\vartheta$
is locally Lipschitz around $\bar x$. In these circumstances,
the condition in $(\ref{in:socDlo})$ can never be satisfied.
On the other hand, whenever the p.h. function $\DHlder{\vartheta}{\bar x}{\cdot}:\R^n\longrightarrow\R$
is sublinear (and hence continuous), then $\Fsubd\vartheta(\bar x)=\partial
\DHlder{\vartheta}{\bar x}{\cdot}(\nullv)\ne\varnothing$.
This happens e.g. (but not only) when $\vartheta:\R^n\longrightarrow\R$ is convex, in which
case one has $\Fsubd\vartheta(\bar x)=\partial\vartheta(\bar x)$.

(ii) The local Lipschitz continuity of $\vartheta$ near $\bar x$ might lead
to believe that the Clarke subdifferential may come into play in the current
context. Recall that the latter is defined by
$$
  \Csubd\vartheta(\bar x)=\left\{v\in\R^n:\ \langle v,w\rangle\le
  \limsup_{x\to\bar x\atop t\downarrow 0}{\vartheta(x+tw)-\vartheta(x)\over t},
  \quad\forall w\in\R^n\right\}.
$$
Since, if $\vartheta$ is locally Lipschitz around $\bar x$, then  it is
$\Fsubd\vartheta(\bar x)\subseteq\Csubd\vartheta(\bar x)$ (see, for instance,
\cite[Chapter 1]{Mord18}), it follows that the condition
\begin{equation}    \label{in:socDClarke}
  \nullv\in\Csubd\vartheta(\bar x)+\inte\left[\dcone{\left(\bigcap_{z\in K}
  \Bder{f}{\bar x,z}^{-1}(\Tang{f(\bar x,z)}{C})\cap\Tang{\bar x}{K}\right)}
  \right]
\end{equation}
does not imply in general the condition in $(\ref{in:socDlo})$.
\end{remark}

\vskip1cm


\end{document}